\documentclass[a4paper,11pt]{amsart}

\usepackage{amsmath,amssymb,amsthm,latexsym,amscd,mathrsfs}
\usepackage{indentfirst}
\usepackage{stmaryrd}
\usepackage{graphicx}
\usepackage{extarrows}
\usepackage{multirow}
\usepackage{latexsym}
\usepackage{amsfonts}
\usepackage{color}
\usepackage{pictexwd,dcpic}
\usepackage{graphicx}
\usepackage{psfrag}
\usepackage{hyperref}
\usepackage{comment}
\usepackage{enumerate}
\usepackage{dutchcal}

\numberwithin{equation}{section} \theoremstyle{plain}
\newtheorem{thm}{Theorem}[section]

\newtheorem{prop}[thm]{Proposition}
\newtheorem{lem}[thm]{Lemma}

\newtheorem{defn}[thm]{Definition}

\newtheorem{prob}[thm]{Problem}

\makeatletter

\def\<{\langle}
\def\>{\rangle}
\def\({\left(}
\def\){\right)}
\def\[{\left[}
\def\]{\right]}

\makeatother

\title{Weakly stable irreducible Yang-Mills fields over $S^4$}

\author[J.Q. Ge]{Jianquan Ge}
\address{School of Mathematical Sciences, Laboratory of Mathematics and Complex Systems, Beijing Normal University, Beijing 100875, P.R. CHINA.}
\email{jqge@bnu.edu.cn}

\author[L.X. Xiao]{Lixin Xiao$^{*}$}
\address{School of Mathematical Sciences, Laboratory of Mathematics and Complex Systems, Beijing Normal University, Beijing 100875, P.R. CHINA.}
\email{lixinxiao@mail.bnu.edu.cn}

\thanks {$^{*}$ the corresponding author.}
\thanks{J. Q. Ge is partially supported by NSFC (No. 12571049) and the Fundamental Research Funds for the Central Universities.}

\begin{document}
\maketitle

\begin{abstract}
Addressing Yau's conjecture (Problem 117) on $S^4$, we investigate the self-duality of weakly stable Yang-Mills fields under the assumption of irreducibility. For structure groups with a simple Lie algebra, we prove that any weakly stable irreducible connection must be either self-dual or anti-self-dual. Furthermore, we demonstrate that if the Lie algebra admits a non-trivial abelian center, no irreducible Yang-Mills fields can exist over $S^4$.
\end{abstract}

\section{Introduction}

This article mainly focuses on Problem 117 in Yau's ``Problem Section" (see \cite{Y14}). 
\begin{prob}
Prove that any Yang-Mills field on $S^{4}$ is either self-dual or anti-self-dual.
\end{prob}
Partial results were given in two papers in 1979 by Bourguignon, Lawson and Simons \cite{BLS79} and in 1981 by Bourguignon and Lawson \cite{BL81} under stability and structure group restrictions. 
\begin{thm}\label{BL}
    \cite{BL81} Any weakly stable Yang-Mills field on $S^{4}$ with structure group $G=SU(2), SU(3)$ or $U(2)$ is either self-dual or anti-self-dual.
\end{thm}
However, when the structure group $G=SO(4)$, \cite{BL81} proved that any weakly stable Yang-Mills field on $S^{4}$ satisfies the so-called two-fold self-duality and there exist many such Yang-Mills fields which are neither self-dual nor anti-self-dual (for example, the Levi-Civita connection on the oriented tangent frame bundle of $S^{4}$). Itoh \cite{I81} showed in 1981 that there exist Yang-Mills fields on $S^{4}$ with any compact simple Lie group $G$ of rank $>3$ which are not weakly stable and thus are not (anti-) self-dual. Sibner, Sibner and Uhlenbeck \cite{SSU89} also showed the existence of non-self-dual Yang-Mills fields on $S^{4}$. Moreover, it was proven by Sadun and Segert \cite{SS91, SS92} that there exist non-stable and non-self-dual Yang-Mills fields for every $SU(2)$ bundle on $S^{4}$ with the second Chern number $c_{2}\not=\pm 1$. Therefore, both the weak stability and structure group restrictions are necessary to deduce the self-duality of Yang-Mills fields on $S^{4}$.

Parallel to these existence results, the analytical properties of the Jacobi operator over $S^m$ have been further refined in terms of Morse index estimates; we refer to Xin \cite{Xin80}, as well as Nayatani-Urakawa \cite{NU95} and Ni \cite{N23}.

The structural classification of energy-minimizing Yang-Mills fields was significantly advanced by Stern \cite{S10}, who proved that on complete, oriented, nonnegatively curved homogeneous Riemannian 4-manifolds, any local minimizer is either an instanton or splits into a sum of instantons upon passage to the adjoint bundle. While this foundational result provides a complete geometric dichotomy for local minimizers, it does not provide explicit algebraic or topological criteria to determine exactly when a minimizer is forced to be an instanton versus when the splitting phenomenon inevitably occurs.

In this article, we show that the simplicity of the Lie algebra, together with its irreducibility, precludes the splitting phenomena in Stern's framework and forces weakly stable Yang-Mills fields to be instantons. Our main result is as follows:

\begin{thm} \label{thm-main}
Let $G$ be a compact connected Lie group with a simple Lie algebra $\mathfrak{g}$, then any weakly stable irreducible Yang-Mills field over $S^4$ with structure group $G$ must be either self-dual or anti-self-dual.
\end{thm}

This theorem comprehensively covers groups such as $SU(n)$ $(n\geq2)$, $SO(n)$ $(n\neq4)$, $Sp(n)$, etc. It is worth noting that the topological simplicity of the base manifold is just as indispensable as the algebraic conditions for deducing this property. To see this, one can look at the recent work of Yin \cite{Y24}. On the manifold $M = \mathbb{CP}^2 \# \mathbb{CP}^2$ equipped with a generic metric, Yin constructed an $SO(3)$ bundle admitting a minimizing Yang-Mills connection that is irreducible yet not an instanton. The existence of such a counterexample is rooted in the non-triviality of the second cohomology group $H^2(M, \mathbb{R})$, which permits the existence of non-trivial harmonic 2-forms. By applying the topological constraints of $S^4$, one finds that a non-trivial center in the Lie algebra, such as for groups like $U(n)$ or $U(n) \times T^k$, precludes the existence of irreducible connections. This phenomenon is summarized in the following proposition:

\begin{prop} \label{no irreducible}
If $\mathfrak{g}$ admits a decomposition $\mathfrak{g} = \mathfrak{g}_1 \oplus \mathfrak{g}_0$, where $\mathfrak{g}_1$ is a simple ideal and $\mathfrak{g}_0 \neq 0$ is the abelian center, then there exists no irreducible Yang-Mills field over $S^4$ with structure group $G$.   
\end{prop}

The remainder of this paper is organized as follows. Section 2 reviews the stability theory and algebraic properties of Lie algebras. Section 3 is devoted to the proofs of Theorem \ref{thm-main} and Proposition \ref{no irreducible}. 

\section{Preliminaries}

In this section, we review the Yang-Mills functional, and the structural theory of connections. We rely on the structural and algebraic properties of Lie groups and algebras as detailed by Hall \cite{H13}.

Let $(S^4, g_0)$ be the standard 4-sphere, $P$ a principal $G$-bundle over $S^4$, and $E$ the associated $G$-vector bundle of rank $r$ over $S^4$, with projections $\pi: P \to S^4$ and $\pi: E \to S^4$, where $G$ is a compact Lie group. Recall for a given faithful representation $\rho: G \to O(r)$, $E$ is given as
\begin{equation*}
E = P \times_{\rho} \mathbb{R}^r = \{ [u, y]; u \in P, y \in \mathbb{R}^r \},
\end{equation*}
where $[u, y]$ is an equivalence class containing $(u, y) \in P \times \mathbb{R}^r$, with the relation $(u, y) \sim (ub, \rho(b)^{-1}y)$, $b \in G$. The bundle $P \times_{\text{Ad}} \mathfrak{g}$ is identified with a subbundle of $\text{End}(E)$ via $\rho$, denoted by $\mathfrak{g}_E$. The identification is given by
\begin{equation*}
P \times_{\text{Ad}} \mathfrak{g} \ni [u, A] \mapsto u \circ \rho(A) \circ u^{-1} \in \text{End}(E).
\end{equation*}

A connection $D$ on $P$ induces a connection on $\mathfrak{g}_E$, which we also denote by $D$. Its curvature $R^D$ is a $\mathfrak{g}_E$-valued 2-form on $S^4$, defined for any vector fields $X, Y \in \mathfrak{X}(S^4)$ by:
\begin{equation*}
R^D(X, Y) = [D_X, D_Y] - D_{[X, Y]}.
\end{equation*}

The bundle $\mathfrak{g}_E$ is equipped with an inner product $\langle \cdot, \cdot \rangle$, induced by a fixed $\operatorname{Ad}(G)$-invariant inner product on $\mathfrak{g}$. For $\mathfrak{g}_E$-valued $p$-forms $\phi, \psi \in \Omega^p(\mathfrak{g}_E)$, the corresponding $L^2$ inner product and norm are defined as:
\begin{equation*}
    (\phi, \psi)_{L^2} = \int_{S^4} \langle \phi, \psi \rangle \, d\mu_{g_0}, \quad \text{and} \quad \|\phi\|^2 = (\phi, \phi)_{L^2}.
\end{equation*}

In general, a connection $D$ induces an exterior covariant derivative $d_D: \Omega^p(\mathfrak{g}_E) \to \Omega^{p+1}(\mathfrak{g}_E)$, $p \ge 0$. Let $d_D^* : \Omega^{p+1}(\mathfrak{g}_E) \to \Omega^p(\mathfrak{g}_E)$, $p \ge 0$, be the formal adjoint of the operator $d_D$. (see \cite{BL81})

The Yang-Mills functional $\mathscr{Y\!\!M}$ is defined as the $L^2$ energy of the curvature $R^D$:
\begin{equation*}
    \mathscr{Y\!\!M}(D) = \frac{1}{2} \|R^{D}\|^{2}.
\end{equation*}

A connection $D$ is called a Yang-Mills connection if it is a critical point of $\mathscr{Y\!\!M}$, which is equivalent to satisfying the Euler-Lagrange equation:
\begin{equation} \label{YM_eqn}
    d_D^* R^D = 0.
\end{equation}

On an oriented 4-manifold, the Hodge star operator $*: \Omega^2(\mathfrak{g}_E) \to \Omega^2(\mathfrak{g}_E)$ satisfies $*^2 = \operatorname{id}$. This induces an orthogonal decomposition of the curvature $R^D \in \Omega^2(\mathfrak{g}_E)$ into self-dual and anti-self-dual components:
\begin{equation*}
    R^D = R^+ + R^-, \quad \text{satisfying} \quad *R^\pm = \pm R^\pm.
\end{equation*}

A connection is said to be self-dual (or anti-self-dual) if its curvature satisfies $R^- = 0$ (or $R^+ = 0$). Such connections are referred to as instantons.

To establish the link between pointwise curvature and the holonomy group, we recall the holonomy algebra and the irreducibility condition.
\begin{thm}[Ambrose-Singer \cite{AS53}]\label{thm:AS}
    Let $P$ be a principal $G$-bundle over $S^4$ and $D$ be a connection on $P$. Let $\Phi(u)$ be the holonomy group at a reference point $u \in P$. The Lie algebra of $\Phi(u)$, which we subsequently denote by $\mathfrak{hol}_x$ at the base point $x \in S^4$, is exactly the subspace of $\mathfrak{g}$ spanned by all elements of the form $R^D_{v}(V, W)$, where $v$ lies in the holonomy bundle through $u$, and $V, W$ are arbitrary horizontal vectors.
\end{thm}

\begin{defn} \label{d1}
A connection $D$ is irreducible if its holonomy group equals the entire structure group $G$. Consequently, the holonomy algebra is precisely $\mathfrak{g}$.
\end{defn}

For a variation field $B \in \Omega^1(\mathfrak{g}_E)$, the second variation is determined by the Jacobi operator (stability operator) $\mathfrak{F}^D$. It has been extensively studied in \cite{BL81,NU95,N23,Xin80}, and is defined as:
\begin{equation*}
    \mathfrak{F}^{D}(B) = \Delta_{d_{D}} B + \mathcal{R}^{D}(B) = D^{*}DB + B \circ \operatorname{Ric} + 2\mathcal{R}^{D}(B),
\end{equation*}
where $\Delta_{d_D} = d_D d_D^* + d_D^* d_D$ is the Hodge-Laplacian, and $D^*D$ is the connection Laplacian. The curvature-related endomorphism $\mathcal{R}^{D}(B)$ acting on any vector field $X \in \mathfrak{X}(S^4)$ is given by:
\begin{equation*}
    \mathcal{R}^{D}(B)(X) = \sum_{j=1}^{4} [R^D_{e_j, X}, B(e_j)],
\end{equation*}
with $\{e_1, \dots, e_4\}$ forming an orthonormal basis of $T_x S^4$.

A connection $D$ is weakly stable if the Jacobi operator $\mathfrak{F}^D$ is non-negative:
\begin{equation*}
    (\mathfrak{F}^D B, B)_{L^2} = \int_{S^4} \langle \mathfrak{F}^D B, B \rangle \, d\mu_{g_0} \geq 0
\end{equation*}
for all $B \in \Omega^1(\mathfrak{g}_E)$.

Finally, for general weakly stable Yang-Mills fields, the self-dual and anti-self-dual components exhibit remarkable algebraic decoupling.

\begin{lem} \cite{S10} \label{splitting}
Let $D$ be a weakly stable Yang-Mills connection on a principal $G$-bundle over $S^4$. Let $\mathcal{K}^+_x, \mathcal{K}^-_x \subseteq \mathfrak{g}_{E,x}$ be the Lie subalgebras generated by $R^+$ and $R^-$, defined in the exact same manner as the holonomy algebra $\mathfrak{hol}_x$ is generated by $R^D$ in Theorem \ref{thm:AS}. Then $[\mathcal{K}_x^+, \mathcal{K}_x^-] = 0$ holds for every $x\in S^4$. 
\end{lem}

This commutation property will be useful in the proof of our main result.

\section{Proof of the Main Theorem}

In this section, we prove Theorem \ref{thm-main} and Proposition \ref{no irreducible} by utilizing the Ambrose-Singer theorem, the Hodge theory on $S^4$, and the structural constraints of the Lie algebra.

\begin{proof}[Proof of Theorem \ref{thm-main}] 
    
Let $D$ be an irreducible, weakly stable Yang-Mills connection on a principal $G$-bundle $P$ where $G$ is a simple Lie group, i.e., $\mathfrak{g}$ is a simple Lie algebra. 

$(a)$ $\mathcal{K}_x^+$ and $\mathcal{K}_x^-$ are ideals of $\mathfrak{g}_{E,x}$ and $\mathfrak{g}_{E,x} = \mathcal{K}_x^+ + \mathcal{K}_x^-$.

By Theorem \ref{thm:AS}, the holonomy algebra $\mathfrak{hol}_x$ is generated by the values of $R^D$ and their parallel transports. Since $R^D = R^+ + R^-$, it follows that $\mathfrak{hol}_x \subseteq \mathcal{K}_x^+ + \mathcal{K}_x^-$. Conversely, Lemma \ref{splitting} ensures that
\begin{equation*}
\mathcal{K}_x^+ + \mathcal{K}_x^- \subseteq \mathfrak{g}_{E,x}.
\end{equation*}

By the irreducibility, we have $\mathfrak{hol}_x = \mathfrak{g}_{E,x}$. Therefore, we obtain the exact decomposition $\mathfrak{g}_{E,x} = \mathcal{K}_x^+ + \mathcal{K}_x^-$.

To see that $\mathcal{K}_x^+$ is an ideal in $\mathfrak{g}_{E,x}$, let $X \in \mathfrak{g}_{E,x}$ and $Y \in \mathcal{K}_x^+$. Writing $X = X^+ + X^-$ with $X^\pm \in \mathcal{K}_x^\pm$, we have $[X, Y] = [X^+, Y] + [X^-, Y]$. The first term lies in $\mathcal{K}_x^+$, and the second term vanishes due to the commuting property $[\mathcal{K}_x^+, \mathcal{K}_x^-] = 0$. Thus, $[X, Y] \in \mathcal{K}_x^+$, making $\mathcal{K}_x^+$ (and similarly $\mathcal{K}_x^-$) an ideal of $\mathfrak{g}_{E,x}$.

$(b)$ $D$ is self-dual. 

The simplicity of $\mathfrak{g}$ restricts its only ideals to $\{0\}$ and $\mathfrak{g}$. The decomposition $\mathfrak{g}_{E,x} = \mathcal{K}_x^+ + \mathcal{K}_x^-$ forces exactly one component to be $\mathfrak{g}_{E,x}$ and the other to be $\{0\}$ at each point $x\in S^4$.

Define the sets $U_\pm = \{x \in S^4 \mid \mathcal{K}_x^\pm = \mathfrak{g}_{E,x}\}$. Note that $U_+$ and $U_-$ are open sets, because the continuous distributions $\mathcal{K}_x^\pm$ attain maximal rank on them, a condition equivalent to the non-vanishing of a continuous determinant function. Our pointwise analysis ensures that $U_+ \cap U_- = \emptyset$ and $U_+ \cup U_- = S^4$.
Since $S^4$ is connected, one of these open sets must be empty.
Without loss of generality, assume $U_- = \emptyset$, then $\mathcal{K}_x^- = \{0\}$ everywhere, forcing $R^- \equiv 0$.
Therefore, the connection $D$ is self-dual.

\end{proof}

\begin{proof}[Proof of Proposition \ref{no irreducible}] 
Let $\mathfrak{g} = \mathfrak{g}_1 \oplus \mathfrak{g}_0$ with $\mathfrak{g}_0 \neq 0$.
Suppose there exists an irreducible Yang-Mills connection $D$ over $S^4$.
We can orthogonally decompose its curvature as $R^D = R_{\mathfrak{g}_1}^D + R_{\mathfrak{g}_0}^D$, corresponding to the simple ideal $\mathfrak{g}_1$ and the abelian center $\mathfrak{g}_0$, respectively.

Recall that locally, the exterior covariant derivative acts on a Lie algebra-valued form $\omega$ by $d_D \omega = d\omega + [A \wedge \omega]$, where $A$ is the connection 1-form.
Because $\mathfrak{g}_0$ is the center of the Lie algebra $\mathfrak{g}$, the Lie bracket $[A \wedge R_{\mathfrak{g}_0}^D]$ vanishes identically. Therefore, the exterior covariant derivative $d_D$ acting on $R_{\mathfrak{g}_0}^D$ reduces to the standard exterior derivative $d$. Projecting the second Bianchi identity $d_D R^D = 0$ and the Yang-Mills equation \eqref{YM_eqn} directly onto the central component $\mathfrak{g}_0$ yields the standard exterior derivative and its formal adjoint:
\begin{equation*}
d R_{\mathfrak{g}_0}^D = 0 \quad \text{and} \quad d^* R_{\mathfrak{g}_0}^D = 0.
\end{equation*}
This implies that $R_{\mathfrak{g}_0}^D$ is a $\mathfrak{g}_0$-valued harmonic 2-form.

By the Hodge isomorphism theorem, the space of harmonic $k$-forms on a compact orientable Riemannian manifold is isomorphic to the $k$-th de Rham cohomology group $H^k_{\text{dR}}(M, \mathbb{R})$.
Since $H^2_{\text{dR}}(S^4, \mathbb{R}) = 0$, any such harmonic 2-form must vanish identically.
Therefore, we have $R_{\mathfrak{g}_0}^D \equiv 0$ on $S^4$.

According to Theorem \ref{thm:AS}, the holonomy algebra $\mathfrak{hol}_x$ is generated by the values of the curvature and its parallel transports. 
The vanishing of the central component forces $\mathfrak{hol}_x \subseteq \mathfrak{g}_1$.
However, this strictly contradicts the assumption that $D$ is an irreducible connection, which would require $\mathfrak{hol}_x = \mathfrak{g}$.
Thus, no such irreducible connection can exist over $S^4$.

\end{proof}


\end{document}